\documentclass[10pt]{amsart}

\usepackage{amsmath, amsfonts, amsthm, amssymb, graphicx, fullpage, enumerate}

\newtheorem{theorem}{Theorem}

\newtheorem*{main theorem}{Main Theorem}

\theoremstyle{remark}     
   
\theoremstyle{definition}

\def\N{\mathbb{N}}

\def\R{\mathbb{R}}     
\def\Z{\mathbb{Z}}

\def\norm#1{\|#1\|}

\begin{document}
\title{Sets in $\R^d$ with slow-decaying density that avoid an unbounded collection of distances} 
\author{Alex Rice}
 
\begin{abstract} For any $d\in \N$ and any function $f:(0,\infty)\to [0,1]$ with $f(R)\to 0$ as $R\to \infty$, we construct a set $A \subseteq \R^d$ and a sequence $R_n \to \infty$ such that $\norm{x-y} \neq R_n$ for all $x,y\in A$ and $\mu(A\cap B_{R_n})\geq f(R_n)\mu(B_{R_n})$ for all $n\in \N$, where $B_R$ is the ball of radius $R$ centered at the origin and $\mu$ is Lebesgue measure. This construction exhibits a form of sharpness for a result established independently by Furstenberg-Katznelson-Weiss, Bourgain, and Falconer-Marstrand, and it generalizes to any metric induced by a norm on $\R^d$.

\end{abstract}

\address{Department of Mathematics, Millsaps College, Jackson, MS 39210}
\email{riceaj@millsaps.edu}

\maketitle 
\setlength{\parskip}{5pt}   

\section{introduction}

 For $d\in \N$ and $R>0$, we let $B_R$ denote the standard open ball in $\R^d$ of radius $R$, centered at the origin, and for $A\subseteq \R^d$, we let $A_R$ denote $A\cap B_R$. Further, we let $\mu$ denote Lebesgue measure on $\R^d$, and $\norm{\cdot}$ denote the standard Euclidean norm on $\R^d$. We say that $A\subseteq \R^d$ has \textit{positive upper density} if $\limsup_{R\to \infty} \mu(A_R)/\mu(B_R)>0. $
Bourgain \cite{Bour} showed via harmonic analysis that if $A\subseteq \R^2$ has positive upper density, then $A$ determines all sufficiently large distances, meaning there exists $R_0=R_0(A)$ such that for all $R>R_0$, there exist $x,y\in A$ with $\norm{x-y}=R$. Further, he generalized his proof to establish the analogous result in $\R^d$ with distances replaced by isometric copies of dilates of any $d$-point configuration that is not contained in a single $(d-2)$-dimensional plane. In his paper, Bourgain alludes to the $d=2$ result as having been previously established via ergodic theory by Katznelson and Weiss, and this argument was later published by those two authors in joint work with Furstenberg \cite{KW}. A short geometric proof of the $d=2$ case was also found by Falconer and Marstrand \cite{Falc} around the same time. 

Written contrapositively, this common theorem states that if $A\subseteq \R^2$ misses a sequence of distances tending to infinity, then $\lim_{R\to \infty} \mu(A_R)/\mu(B_R)=0$. One could hope to quantitatively strengthen this conclusion with an estimate of the form $\mu(A_R)/\mu(B_R) \leq f(R)$ for all $R\geq R_0(A)$, where $f(R)\to 0$ independent of $A$, such as $f(R)=1/\log R$. Here we establish that such a quantitative improvement is impossible, and hence the aforementioned results are, in a sense, sharp. In particular, given $d\in \N$ and a function $f:(0,\infty)\to [0,1]$ tending to $0$ as $R\to \infty$, we construct a set $A\subseteq \R^d$ and a sequence $R_n\to \infty$ such that $\norm{x-y}\neq R_n$ for all $x,y\in A$ and $\mu(A_{R_n})\geq f(R_n)\mu(B_{R_n})$ for all $n\in \N$. In addition to working in any dimension $d$, the construction generalizes to any metric induced by a norm on $\R^d$.

\section{Motivation}
Since the set of distances between integer lattice points in $\R^d$, which we denote by $\Delta_d$, is a fairly sparse, discrete set consisting entirely of square roots of integers, a natural candidate for a dense set $A\subseteq \R^d$ with lots of missing distances is a a union of \textit{thickened} lattice points, meaning small balls around lattice points. However, for $d\geq 2$, based on standard results concerning sums of squares, the gaps between consecutive elements of $\Delta_d$ tend to $0$, hence the full integer lattice, thickened at even a single point, determines all sufficiently large distances. Alternatively, since $\Delta_d$ has no limit points, we can certainly thicken a cube $L=[a,b]^d\cap\Z^d$ by a fixed $\epsilon>0$ and avoid a single distance $R\notin \Delta_d$. The relative density of the thickened cube in $[a,b]^d$ is about $\epsilon^d$, and we can choose $L$ to be as large as we want with respect to $\epsilon$. We iterate this process to form a union of cubes, far enough apart so as to control the interactions between them, that avoids a rapidly growing sequence of distances that fail to occur in the lattice. The crucial detail is that, while the closing gaps between elements of $\Delta_d$ force the density of our set to decay to $0$, we have control over the relationship between the relative density and the scale of the cube at each step of the construction.

\section{Construction for Arbitrary Norms}

Recall that $\rho:\R^d \to [0,\infty)$ is called a \textit{norm} if $\rho(x)>0$ for all $x\neq\vec{0}$, $\rho(\lambda x)=|\lambda|\rho(x)$ for all $\lambda \in \R$ and all $x\in \R^d$ (homogoneity), and $\rho(x+y)\leq \rho(x)+\rho(y)$ for all $x,y\in \R^d$ (triangle inequality). It is a standard fact that all norms on $\R^d$ are equivalent, meaning there exist constants $c_{\rho},C_{\rho}>0$ such that \begin{equation}\label{equiv} c_{\rho}\norm{x}\leq \rho(x) \leq C_{\rho}\norm{x} \text{ for all } x\in \R^d. \end{equation}

\noindent For a norm $\rho$ on $\R^d$, we let $B^{\rho}_R=\{ x\in \R^d: \rho(x)<R\}$, and $A^{\rho}_R=A\cap B^{\rho}_R$. Homogeneity of $\rho$ ensures that \begin{equation}\label{homog} \mu(B^{\rho}_R)=\omega_{\rho}R^d, \end{equation} where $\omega_{\rho}=\mu(B^{\rho}_1) \geq c_{\rho}^d\mu(B_1)$. A norm $\rho$ on $\R^d$ also determines a metric on $\R^d$, by defining the \textit{$\rho$-distance} between $x$ and $y$ to be $\rho(x-y)$. Our main result is the following.

\begin{theorem} \label{main} Suppose $f:(0,\infty)\to [0,1]$, $d\in \N$, and $\rho$ is a norm on $\R^d$. If $f(R)\to 0$ as $R\to \infty$, then there exists $A\subseteq \R^d$ and a sequence of positive numbers $R_n\to \infty$ such that:

\begin{itemize} \item[(i)] $\rho(x-y)\neq R_n$ for all $x,y\in A$ and all $n\in \N$. \\ \item[(ii)] $\mu(A^{\rho}_{R_n}) \geq f(R_n)\mu(B^{\rho}_{R_n})$ for all $n\in \N$.

\end{itemize}

\end{theorem}

\begin{proof} Fix $d\in \N$, $f:(0,\infty)\to [0,1]$ with $f(R)\to 0$ as $R\to \infty$, a norm $\rho$ on $\R^d$, and constants $c_{\rho},C_{\rho}>0$ satisfying (\ref{equiv}). Let $1=R_0<R_1<R_2<\cdots$ and $1=\epsilon_0\geq\epsilon_1\geq\epsilon_2\geq\cdots$ be any sequences satisfying \begin{enumerate}[(a)]  \item $|\rho(x)-R_{j}|>\epsilon_{n}$ for all $x\in \Z^d$ and all $1\leq j\leq n$, \\ \item $R_n \geq 100 R_{n-1}$, and \\ \item $f(R_n) \leq \left(\frac{\epsilon_{n-1}}{16C_{\rho}}\right)^d$ \end{enumerate} for all $n\in \N$. We note that property (a) is achievable because, for any interval $[a,b]\subseteq [0,\infty)$, $$\{x\in \Z^d: \rho(x)\in [a,b]\} \subseteq \{x\in \Z^d: \norm{x}\in [a/C_{\rho},b/c_{\rho}]\}  $$ is finite, hence $\rho(\Z^d)$ is a discrete set with no limit points. 

\noindent Let $L_n$ be a translate of $\{0,1,\dots,\lfloor R_n/(4C_{\rho}) \rfloor\}^d$ lying in $B^{\rho}_{R_n/2} \setminus B^{\rho}_{10R_{n-1}}$, which exists by (b), let $$P_n=\{y\in \R_d: \rho(x-y)<\epsilon_{n-1}/4 \text{ for some }x\in L_n\},$$ and let $A=\bigcup_{n=1}^{\infty} P_n.$ For condition (ii), we see that $P_n \subseteq B^{\rho}_{R_n}$, hence $$\mu(A^{\rho}_{R_n})\geq \mu(P_n) \geq \omega_{\rho}\left(\frac{\epsilon_{n-1}}{16C_{\rho}}\right)^{d}R_n^d\geq f(R_n)\mu(B^{\rho}_{R_n}),$$ where the second inequality uses that $P_n$ is a disjoint union of at least $[R_n/(4C_{\rho})]^d$ $\rho$-balls of radius $\epsilon_{n-1}/4$, and the third inequality comes from (c) and (\ref{homog}). 

\noindent Further, we argue inductively that, for each $n\in \N$, $\rho(x-y)\neq R_j$ for all $x,y \in A^{\rho}_{R_n}$ and all $j\in \N$, which suffices to establish (i).  $A^{\rho}_{R_n}\subseteq B^{\rho}_{R_n/2}$, so $A^{\rho}_{R_n}$ certainly has no $R_j$ $\rho$-distances for $j\geq n$, which in particular establishes the base case $n=1$. We now fix $n\geq 2$ and suppose $A^{\rho}_{R_{n-1}}$ has no $R_j$ $\rho$-distances for all $j\in \N$, and we wish to say the same for $A^{\rho}_{R_n}=A^{\rho}_{R_{n-1}}\cup P_n$. Since $A^{\rho}_{R_{n-1}}$ lies inside $B^{\rho}_{R_{n-1}}$ and $P_n$ lies outside $B^{\rho}_{10R_{n-1}}$, the $\rho$-distance between any point in $A^{\rho}_{R_{n-1}}$ and any point in $P_n$ is at least $9R_{n-1}$, so it suffices to show that for every $n\in \N$, $\rho(x-y)\neq R_j$ for all $x,y \in P_n$ and all $1\leq j\leq n-1$. 

\noindent By the triangle inequality, the $\rho$-distances between points in $P_n$ are all within $\epsilon_{n-1}/2$ of $\rho$-distances between points in $L_n$. However, by (a), no such distances can equal $R_j$ for $1\leq j\leq n-1$, and property (i) follows. \end{proof}

\section{Concluding Remarks}

It is worth noting that Theorem \ref{main} is only informative for a particular norm if the aforementioned positive result, that sets of positive upper density determine all sufficiently large distances, holds with respect to that norm. For two examples where the result fails, if $\rho$ is the $\ell^{\infty}$ or $\ell^{1}$ norm, then $B^{\rho}_1$ is a $2d$ or $2^d$ faced polytope, respectively, and $\rho(\Z^d)$ is the set of nonnegative integers, so one can take $A$ to be the full integer lattice thickened by $\epsilon=1/8$, and $A$ has positive density but misses all half-integer distances. Kolountzakis \cite{Kol} showed that this type of example is the only obstruction, establishing that the positive result holds for a norm $\rho$ on $\R^d$ provided $B^{\rho}_1$ is not a polytope, expressly because an analog of the construction from the previous sentence, a thickening of a well-distributed set with a separated collection of $\rho$-distances, is impossible. It is not known when exactly such a construction is possible in the event that $B^{\rho}_1$ is a polytope, so the full converse of the Kolountzakis result is still open.  Whatever the precise collection of norms may be, Theorem \ref{main} ensures that whenever the positive result holds, it is sharp in the sense discussed in the introduction.  

%\noindent \textbf{Note from Alex R.:} In a sense, this is more general than we care about, because there are norms for which the positive Katznelson-Weiss result fails, such as the $\ell^1$ norm inducing the taxicab metric. However, I kind of like that if you go through the proof carefully, you can see that in certain cases, the construction will actually give you a set with positive density. Specifically, this amounts to the question of whether you really need the sequence $\epsilon_n$ in the proof to tend to zero. For the taxicab metric, one could just take $R_n$ to be half-integers and take $\epsilon_n=1/4$ for all $n$, and you get positive density. Morally, the need for $\epsilon_n$ to decrease seems to be detecting smoothness/curvature?

\noindent \textbf{Acknowledgements:} The author would like to thank Alex Iosevich for posing this problem at the 2019 Combinatorial and Additive Number Theory (CANT) conference, held at the CUNY Graduate Center in memory of Jean Bourgain. The author would also like to thank Neil Lyall for his helpful discussions.

\end{document}